\documentclass[12pt,reqno]{amsart}
\usepackage{amssymb}
\usepackage[latin1]{inputenc}
\usepackage{dsfont}
\usepackage{hyperref}
\usepackage{fullpage}

\newtheorem{theorem}{Theorem}[section]
\newtheorem{lemma}[theorem]{Lemma}
\newtheorem{proposition}[theorem]{Proposition}

\theoremstyle{definition}
\newtheorem{definition}[theorem]{Definition}

\theoremstyle{remark}
\newtheorem{remark}[theorem]{Remark}

\numberwithin{equation}{section}

\begin{document}

\title[On penalisation results related with a remarkable class of submartingales]
{On penalisation results related with a remarkable class of submartingales}
\author[J. Najnudel]{Joseph Najnudel}
\address{Institut f\"ur Mathematik, Universit\"at Z\"urich, Winterthurerstrasse 190,
8057-Z\"urich, Switzerland}
\email{\href{mailto:joseph.najnudel@math.uzh.ch}{joseph.najnudel@math.uzh.ch}}

\author[A. Nikeghbali]{Ashkan Nikeghbali}
\email{\href{mailto:ashkan.nikeghbali@math.uzh.ch}{ashkan.nikeghbali@math.uzh.ch}}

\date{\today}

\begin{abstract}
Is this paper we study penalisations of diffusions satisfying some technical conditions, 
generalizing a result obtained by Najnudel, Roynette and Yor in \cite{NRY}. 
If one of these diffusions has probability distribution $\mathbb{P}$, then our result can 
be described as follows: for a large class of families of probability measures $(\mathbb{Q}_t)_{t \geq 0}$,
each of them being absolutely continuous with respect to $\mathbb{P}$,  there exists
a probability $\mathbb{Q}_{\infty}$ such that for all events $\Lambda$ depending only 
on the canonical trajectory up to a fixed time, $\mathbb{Q}_t (\Lambda)$ tends to $\mathbb{Q}_{\infty} (\Lambda)$
when $t$ goes to infinity. In the cases we study here, the limit measure $\mathbb{Q}_{\infty}$ is 
absolutely continous with respect to a sigma-finite measure $\mathcal{Q}$, which does not depend
on the choice of the family of probabilities $(\mathbb{Q}_t)_{t \geq 0}$, but only on $\mathbb{P}$.
The relation between $\mathbb{P}$ and $\mathcal{Q}$ is obtained in a very general framework
by the authors of this paper in \cite{NN1}.
\end{abstract} 

\maketitle

\section{Introduction}
In a series of articles by Roynette, Vallois and Yor, summarized in \cite{RVY}, 
the authors study many examples of probability measures on functional spaces,
which are obtained as weak limits of measures which are absolutely continuous with respect to a given probability.
The setting generally used is the following: one considers $\mathbb{W}$, the Wiener measure 
on the space of
continuous functions from $\mathbb{R}_+$ to $\mathbb{R}$, denoted $\mathcal{C}(\mathbb{R}_+, \mathbb{R})$
and endowed with its canonical filtration 
$(\mathcal{F}_s)_{s \geq 0}$ (not completed). One defines the $\sigma$-algebra $\mathcal{F}$
by $$\mathcal{F}:= 
\underset{s \geq 0}{\bigvee} \mathcal{F}_s.$$ 
One then considers
 $(\Gamma_t)_{t \geq 0}$, a family of 
nonnegative random variables on the same space such that 
$$0 < \mathbb{W} [\Gamma_t]  < \infty,$$
and for $t \geq 0$, one defines the probability measure $$\mathbb{Q}_t := 
\frac{\Gamma_t}{\mathbb{W}[\Gamma_t]} \, . \mathbb{W}.$$
(In this paper, if $\mathbb{P}$ is a probability measure and $Y$ a random variable, we denote by $\mathbb{P}[Y]$
the expectation of $Y$ with respect to $\mathbb{P}$). Under these assumptions, Roynette, Vallois and Yor have 
proved that for many examples of families of functionals $(\Gamma_t)_{t \geq 0}$, there exists a probability 
measure $\mathbb{Q}_{\infty}$ which can be considered as the weak limit of $(\mathbb{Q}_t)_{t \geq 0}$ 
when $t$ goes to infinity, in the following sense: for all $s \geq 0$ and for all events 
$\Lambda_s \in \mathcal{F}_s$,
$$\mathbb{Q}_t [\Lambda_s] \underset{t \rightarrow \infty}{\longrightarrow} \mathbb{Q}_{\infty} [\Lambda_s].$$
For example, the measure $\mathbb{Q}_{\infty}$ exists for the following families of functionals
$(\Gamma_t)_{t \geq 0}$:
\begin{itemize}
\item $\Gamma_t = \phi(L_t)$, where $(L_t)_{t \geq 0}$ is the local time at zero of the canonical process $X$, 
and $\phi$ is a nonnegative, integrable function from $\mathbb{R}_+$ to $\mathbb{R}_+$.
\item $\Gamma_t = \phi(S_t)$, where $S_t$ is the supremum of $X$ on the interval $[0,t]$, and $\phi$ is, again, 
a nonnegative, integrable function from $\mathbb{R}_+$ to $\mathbb{R}_+$. 
\item $\Gamma_t = e^{-\int_0^t q(X_s) ds}$, where $q$ is a measurable function from $\mathbb{R}$ to 
$\mathbb{R}_+$, such that: $$0 < \int_{\mathbb{R}} (1+|x|) q(x) \, dx < \infty.$$
\item $\Gamma_t = e^{\lambda L_t + \mu |X_t|}$, where $(L_t)_{t \geq 0}$ is, again, the local time at zero 
of $X$. 
\end{itemize}
\noindent
There are still other interesting particular cases which can be studied. For instance in \cite{Na} Najnudel
has proved that the limit measure exists for 
$$\Gamma_t = \exp \left( - \int_{\mathbb{R}} (L_t^y)^2 \, dy \right)$$
where $(L_t^y)_{t \geq 0, y \in \mathbb{R}}$ is the regular family of local times of $X$. This example can be 
interpreted as the construction of a one-dimensional self-avoiding Brownian motion, and hence, a one-dimensional 
version of Edwards' model, for polymers of infinite length (this model was studied with several points of
view: see for example \cite{Edw}, \cite{West}, \cite{Bolt}, \cite{HHK}). Another family of penalisation of probability measures on general functional spaces has been introduced by the authors of this paper in \cite{NN3}, where the functional $(\Gamma_t)$ is of the form $\Gamma_t=F_t X_t$, for a large class of functionals $(F_t)$.
Note that in all these examples, the proof of the existence of $\mathbb{Q}_{\infty}$  
given by the authors cited above contains two steps which need to be clarified here. 
The first point is that the functional $\Gamma_t$ is, in general, not defined everywhere but only 
almost everywhere. For example, since $(\mathcal{F}_s)_{s \geq 0}$ is not completed, there does not 
exist a c\`adl\`ag and adapted version of the local time which is defined everywhere (a more detailed 
discussion of this problem is given, for example, in our article \cite{NN2}). However, the almost 
sure existence of $\Gamma_t$ is sufficient to define the measure $\mathbb{Q}_t$, and to study its 
weak convergence towards $\mathbb{Q}_{\infty}$. The second point is that the existence of $\mathbb{Q}_{\infty}$ 
depends on the possibility one has to extend compatible families of probability measures. More precisely, 
the penalisation results cited above are proved as follows: by studying the asymptotics (for fixed $s \geq 0$ 
and for $t$ going to infinity) of a suitable version of the conditional expectation
 of $\Gamma_t$ given $\mathcal{F}_s$, one proves that for all $s \geq 0$, there exists an 
$\mathcal{F}_s$-measurable, nonnegative random variable $M_s$ such that for all events $\Lambda_s
 \in \mathcal{F}_s$:
\begin{equation}
\mathbb{Q}_{t} (\Lambda_s) \underset{t \rightarrow \infty}{\longrightarrow}
 \mathbb{W} (M_s \mathds{1}_{\Lambda_s}). 
\end{equation}
One immediately deduces that $(M_s)_{s \geq 0}$ is a martingale, and that there exists a compatible 
family $(\mathbb{Q}_{\infty}^{(s)})_{s \geq 0}$ of probability measures, 
$\mathbb{Q}_{\infty}^{(s)}$ defined on $\mathcal{F}_s$, such that for $s \geq 0$, $\Lambda_s \in \mathcal{F}_s$:
$$\mathbb{Q}_{t} (\Lambda_s) \underset{t \rightarrow \infty}{\longrightarrow}
 \mathbb{Q}_{\infty}^{(s)} (\Lambda_s).$$
Hence, if one can construct a probability measure on the space
$(\mathcal{C} (\mathbb{R}_+, \mathbb{R}), \mathcal{F})$, such that its restriction to 
$\mathcal{F}_s$ is $\mathbb{Q}_{\infty}^{(s)}$, the existence of $\mathbb{Q}_{\infty}$ is proved. 
This possibility of extension of measures is not obvious at all: for example, as explained in \cite{NN2}, 
$\mathbb{Q}_{\infty}$ does not exist in general if one replaces the filtered 
probability space $(\mathcal{C}(\mathbb{R}_+, \mathbb{R}), \mathcal{F}, (\mathcal{F}_s)_{s \geq 0}, 
\mathbb{W})$ by its usual augmentation. However, in the setting described above, 
$\mathbb{Q}_{\infty}$ exists, because the filtered measurable space
 $(\mathcal{C} (\mathbb{R}_+, \mathbb{R}), \mathcal{F}, (\mathcal{F}_s)_{s \geq 0})$ satisfies the property
(P), described as follows:
\begin{definition}
 \label{P}
Let $(\Omega, \mathcal{F}, (\mathcal{F}_s)_{s \geq 0})$ be a filtered measurable space, such that
$\mathcal{F}$ is the $\sigma$-algebra generated by $\mathcal{F}_s$, $s \geq 0$: 
$\mathcal{F}=\bigvee_{s\geq0}\mathcal{F}_s$. We say that the property (P)
 holds if and only if $(\mathcal{F}_s)_{s \geq 0}$ enjoys the following conditions: 
\begin{itemize}
\item For all $s \geq 0$, $\mathcal{F}_s$ is generated by a countable number of sets.
\item For all $s \geq 0$, there exists a Polish space $\Omega_s$, and a surjective map 
 $\pi_s$ from $\Omega$ to $\Omega_s$, such that $\mathcal{F}_s$ is the $\sigma$-algebra of the inverse
 images, by $\pi_s$, of Borel sets in $\Omega_s$, and such that for all $B \in \mathcal{F}_s$, 
 $\omega \in \Omega$, $\pi_s (\omega) \in \pi_s(B)$ implies $\omega \in B$.
\item If $(\omega_n)_{n \geq 0}$ is a sequence of elements of $\Omega$, such that for all $N \geq 0$,
$$\bigcap_{n = 0}^{N} A_n (\omega_n) \neq \emptyset,$$
where $A_n (\omega_n)$ is the intersection of the sets in $\mathcal{F}_n$ containing $\omega_n$, 
then:
$$\bigcap_{n = 0}^{\infty} A_n (\omega_n) \neq \emptyset.$$
\end{itemize}
\end{definition}
\noindent
This definition is given in \cite{NN2}, but the corresponding conditions are not new: they are
already stated by Parthasarathy in \cite{Parth}, p. 141. A fundamental example 
of filtered measurable space $(\Omega, \mathcal{F}, (\mathcal{F}_s)_{s \geq 0})$
satisfying the property (P) is the following: for some integer $d \geq 1$, $\Omega$ is 
the space of continuous functions from $\mathbb{R}_+$ to $\mathbb{R}^d$, or the space of c\`adl\`ag
functions from $\mathbb{R}_+$ to $\mathbb{R}^d$, for all $s \geq 0$, $\mathcal{F}_s$ is 
the $\sigma$-algebra generated by the canonical process up to time $s$, and 
 $\mathcal{F}$ is the $\sigma$-algebra generated by $(\mathcal{F}_s)_{s \geq 0}$. This example 
proves that in the examples studied by Najnudel, Roynette, Vallois and Yor, the measure 
$\mathbb{Q}_{\infty}$ can be constructed. Indeed, one has the following result, proved in \cite{NN2},
with methods coming from Stroock and Varadhan (see \cite{SV}):
\begin{proposition} \label{extension}
Let $(\Omega, \mathcal{F}, (\mathcal{F}_s)_{s \geq 0})$ be a filtered measurable space satisfying the 
property (P), and let, for $s \geq 0$, $\mathbb{Q}_s$ be a probability measure on $(\Omega, \mathcal{F}_s)$, 
such that for all $t \geq s \geq 0$, $\mathbb{Q}_s$ is the restriction of $\mathbb{Q}_t$ to $\mathcal{F}_s$.
Then, there exists a unique measure $\mathbb{Q}_{\infty}$ on $(\Omega, \mathcal{F})$ such that for all $s \geq 0$,
its restriction to $\mathcal{F}_s$ is equal to $\mathbb{Q}_s$. 
\end{proposition}
\noindent
All the examples of penalisation results described above involve the Wiener space. In this 
paper, we need to deal with a more general setting, for which it is important to have, at the same 
time, the classical results of stochastic calculus generally proved under usual conditions, and the
possibility of extending compatible families of probability measures. Since this extension is 
in general impossible under usual conditions, we need to complete more carefully the probability spaces. 
The good way to do this completion, intermediate between the right-continuous version and the 
usual augmentation, involves the so-called natural conditions or N-usual conditions. They were first introduced 
by Bichteler in \cite{B}, and then rediscovered in \cite{NN2} where it is shown that most of the properties which are generally proved 
under usual conditions remain true under natural conditions (for example, existence of c\`adl\`ag versions
of martingales, Doob-Meyer decomposition, d\'ebut theorem, etc.). Let us recall here the definition:
\begin{definition} \label{natural}
A filtered probability space $(\Omega,\mathcal{F}, (\mathcal{F}_s)_{s \geq 0}, \mathbb{P})$,
  satisfies the natural conditions iff the two following assumptions hold:
\begin{itemize}
\item The filtration $(\mathcal{F}_s)_{s \geq 0}$ is right-continuous;
\item For all $s \geq 0$, and for every $\mathbb{P}$-negligible set $A \in \mathcal{F}_s$, all
the subsets of $A$ are contained in $\mathcal{F}_0$.
\end{itemize}
\end{definition}
\noindent
This definition is slightly different from the definitions given in \cite{B} and \cite{NN2} but one can 
easily check that it is equivalent. 
The natural enlargement of a filtered probability space can be defined by using the following proposition:
\begin{proposition}[\cite{NN2}]
Let $(\Omega, \mathcal{F}, (\mathcal{F}_s)_{s \geq 0}, \mathbb{P})$ be a filtered probability space.
There exists a unique filtered probability space $(\Omega, \widetilde{\mathcal{F}}, 
 (\widetilde{\mathcal{F}}_s)_{s \geq 0},
\widetilde{\mathbb{P}})$ (with the same set $\Omega$), such that:
\begin{itemize}
\item For all $s \geq 0$, $\widetilde{\mathcal{F}}_s$ contains $\mathcal{F}_s$, $\widetilde{\mathcal{F}}$ 
contains $\mathcal{F}$ and $\widetilde{\mathbb{P}}$ is an extension of $\mathbb{P}$;
\item The space $(\Omega, \widetilde{\mathcal{F}}, 
 (\widetilde{\mathcal{F}}_s)_{s \geq 0}, \widetilde{\mathbb{P}})$ satisfies the natural conditions;
\item For any filtered probability space $(\Omega, \mathcal{F}', (\mathcal{F}'_s)_{s \geq 0}, \mathbb{P}')$
satisfying the two items above, $\mathcal{F}'_s$ contains $\widetilde{\mathcal{F}}_s$ for all $s \geq 0$,
$\mathcal{F}'$ contains $\widetilde{\mathcal{F}}$ and $\mathbb{P}'$ is an extension of 
$\widetilde{\mathbb{P}}$.  
\end{itemize}
\noindent
The space $(\Omega, \widetilde{\mathcal{F}}, 
 (\widetilde{\mathcal{F}}_s)_{s \geq 0}, \widetilde{\mathbb{P}})$ is called the natural enlargement of
 $(\Omega, \mathcal{F},  (\mathcal{F}_s)_{s \geq 0}, \mathbb{P})$.
  \end{proposition}
\noindent
Intuitively, the natural enlargement of a filtered probability space is its smallest extension which 
satisfies the natural conditions. Now, if we combine the natural enlargement with the property (P),
we obtain the following definition:
\begin{definition} \label{NP}
Let $(\Omega, \mathcal{F},(\mathcal{F}_s)_{s \geq 0}, \mathbb{P})$ be a filtered probability space. 
We say that it satisfies the property (NP) iff it is the natural enlargement of a 
filtered probability space $(\Omega, \mathcal{F}^0,(\mathcal{F}^0_s)_{s \geq 0}, \mathbb{P}^0)$
such that the filtered measurable space $(\Omega, \mathcal{F}^0,(\mathcal{F}^0_s)_{s \geq 0})$ enjoys 
property (P). 
\end{definition}
\noindent
The following result about extension of probability measures is proved in \cite{NN2} (in a slightly 
more general form): 
\begin{proposition}[\cite{NN2}] \label{extensionaugmentation}
Let $(\Omega, \mathcal{F}, (\mathcal{F}_s)_{s \geq 0}, \mathbb{P})$ be a filtered probability space, 
satisfying property (NP). 
 Then, the $\sigma$-algebra $\mathcal{F}$ is the $\sigma$-algebra generated by $(\mathcal{F}_s)_{s \geq 0}$,
and for all coherent families of probability measures
$(\mathbb{Q}_s)_{s \geq 0}$, such that $\mathbb{Q}_s$ is defined 
on $\mathcal{F}_s$,
and is absolutely continuous with respect to the restriction
of $\mathbb{P}$ to $\mathcal{F}_s$, there exists a unique probability measure
$\mathbb{Q}$
on $\mathcal{F}$ which coincides with $\mathbb{Q}_s$ on $\mathcal{F}_s$ 
for all $s \geq 0$. 
\end{proposition}
\noindent
The possibility of extension of coherent families of probability measures, on spaces satisfying the property 
(NP) implies that one can obtain penalisation results in this framework. For the functionals 
we shall study in this article, the limit measure is absolutely continuous with respect to 
a remarkable $\sigma$-finite measure constructed in a very general setting in \cite{NN1}. 
This $\sigma$-finite measure was already encountered in different special cases (see 
\cite{BeYor}, \cite{CNP}, \cite{MRY} and \cite{NRY} for example), and it involves a remarkable class 
of submartingales, called $(\Sigma)$. The submartingales of class $(\Sigma)$ were first
 introduced by Yor in \cite{Y}, and some of their main properties were studied in \cite{N}. Let us 
recall the definition:
\begin{definition}
Let $(\Omega,\mathcal{F}, (\mathcal{F}_s)_{s \geq 0},\mathbb{P})$ be a filtered probability space. A 
nonnegative submartingale (resp. local submartingale) $(X_s)_{s \geq 0}$ is of class $(\Sigma)$, if and only if it can 
be decomposed as
$X_s = N_s + A_s$ where $(N_s)_{s \geq 0}$ and $(A_s)_{s \geq 0}$ are $(\mathcal{F}_s)_{s \geq 0}$-adapted 
processes satisfying the following assumptions:
\begin{itemize}
\item $(N_s)_{s \geq 0}$ is a c\`adl\`ag martingale (resp. local martingale);
\item $(A_s)_{s \geq 0}$ is a continuous increasing process, with $A_0 = 0$;
\item The measure $(dA_s)$ is carried by the set $\{s \geq 0, X_s = 0 \}$.
\end{itemize}
\end{definition}
\noindent
In \cite{NN1}, it is proved that for any submartingale of class $(\Sigma)$ defined on a space 
satisfying property (NP), on can associate with it a $\sigma$-finite measure as follows:
\begin{theorem} \label{all}
Let $(X_s)_{s \geq 0}$ be a submartingale of the class $(\Sigma)$ (in particular $X_s$ is integrable 
for all $s \geq 0$), defined on a filtered probability space 
$(\Omega, \mathcal{F}, (\mathcal{F}_s)_{s \geq 0}, \mathbb{P})$ which satisfies the property (NP).
In particular $(\mathcal{F}_s)_{s \geq 0}$
 satisfies the natural conditions and $\mathcal{F}$ is 
the $\sigma$-algebra generated by $\mathcal{F}_s$, $s \geq 0$. 
Then, there exists a unique $\sigma$-finite measure $\mathcal{Q}$, defined on
 $(\Omega, \mathcal{F}, \mathbb{P})$,
such that for $g:= \sup\{s \geq 0, X_s = 0 \}$:
\begin{itemize}
\item $\mathcal{Q} [g = \infty] = 0$;
\item For all $s \geq 0$, and for all $\mathcal{F}_s$-measurable, bounded random variables $F_s$,
$$\mathcal{Q} \left[ F_s \, \mathds{1}_{g \leq s} \right] = \mathbb{P} \left[F_s X_s \right].$$
\end{itemize}
\noindent
\end{theorem} 
\noindent
In this article, we prove that there is a link between penalisations and the measure $\mathcal{Q}$ involved in 
Theorem \ref{all}. More precisely, we show that if $(\Omega, \mathcal{F}, (\mathcal{F}_s)_{s \geq 0},
 \mathbb{P})$ is the natural augmentation of the space $\mathcal{C}(\mathbb{R}_+, \mathbb{R})$ endowed with 
its canonical filtration and the law of a diffusion satisfying some technical conditions,
if $(X_s)_{s \geq 0}$ is the canonical process, 
and if $(\Gamma_t)_{t \geq 0}$ is a nonincreasing family of 
nonnegative random variables on the same space, tending to $\Gamma_{\infty}$ when $t$ goes to 
infinity, and such that 
$$0 < \mathbb{P} [\Gamma_t]  < \infty$$
and
$$0 < \mathcal{Q} [\Gamma_{\infty}] < \infty,$$
then under some precise assumptions stated below, the probability measure $$\mathbb{Q}_t := 
\frac{\Gamma_t}{\mathbb{P}[\Gamma_t]} \, . \mathbb{P},$$
converges weakly (in the sense of penalisations) to the limit
$$\mathbb{Q}_{\infty} := \frac{\Gamma_{\infty}}{\mathcal{Q}[\Gamma_{\infty}]} \, . \mathcal{Q}.$$
\noindent
In the Brownian case, a similar penalisation result, involving the measure $\mathcal{Q}$, is obtained
in \cite{NRY}.

The present paper is organized as follows: in Section \ref{setting}, we detail the precise setting under 
which our penalisation result is available and we state precisely this result, in 
Section \ref{sturm}, we give some estimates of hitting times of the canonical diffusion, and in 
Section \ref{proof}, we use these estimates in order to complete the proof of our main theorem.

\section{The general setting} \label{setting}
The general framework in which one can state our penalisation result was introduced by Salminen, Vallois
and Yor in \cite{SVY}, in a slightly different and
more general form, and it was also used by Najnudel, Roynette and Yor in \cite{NRY}, Chapter 3. 
Note that in both cases the filtration is not completed. In this present paper, we use the 
natural augmentation, which allows us to define the local time everywhere, and to apply Theorem 
\ref{all}.

The underlying filtered probability space is constructed as follows. Let $\Omega$
 be the space of continuous functions from $\mathbb{R}_+$ to 
$\mathbb{R}_+$, $(\mathcal{F}^0_s)_{s \geq 0}$, the natural filtration of $\Omega$ (not completed), and 
$\mathcal{F}^0$, the $\sigma$-algebra generated by $(\mathcal{F}^0_s)_{s \geq 0}$. We assume that the following holds:
 \begin{itemize}
\item  the probability $\mathbb{P}^0$, defined on $(\Omega, \mathcal{F}^0)$, is such that under $\mathbb{P}^0$, the canonical process is a recurrent diffusion in natural scale, starting from a fixed  point $x_0 \geq 0$, with zero as an instantaneously reflecting barrier;
\item the speed measure $m(x)$ of this diffusion is absolutely continuous with respect to Lebesgue measure on $\mathbb{R}_+$, with a continuous density 
$m : \mathbb{R}^*_+ \rightarrow \mathbb{R}^*_+$;
\item we assume that $m(x)$ is 
equivalent to $cx^{\beta}$ when $x$ goes to infinity, for some $c > 0$ and $\beta > -1$, and
we suppose that there exists $C > 0$ such that for all $x > 0$, $m(x) \leq Cx^{\beta}$ if $\beta \leq 0$, and 
$m(x) \leq C(1+x^{\beta})$ if $\beta > 0$;
\item the filtered probability space $(\Omega, \mathcal{F}, (\mathcal{F}_s)_{s \geq 0}, \mathbb{P})$ is 
constructed as the natural augmentation of the space 
$(\Omega, \mathcal{F}^0, (\mathcal{F}^0_s)_{s \geq 0}, \mathbb{P}^0)$:
$(\Omega, \mathcal{F}, (\mathcal{F}_s)_{s \geq 0}, \mathbb{P})$ satisfies the property (NP) and
under $\mathbb{P}$, the law of the canonical process is a diffusion with the same parameters as 
under $\mathbb{P}^0$.
\end{itemize}
This diffusion is in natural scale and, as in \cite{NRY}, one can deduce
 that  the canonical process $(X_s)_{s \geq 0}$ is a submartingale of class $(\Sigma)$. In 
particular, $X_s$ is integrable
 for all $s \geq 0$, moreover,  the local time $(L_s)_{s \geq 0}$ of $(X_s)_{s \geq 0}$ at level zero, is its
increasing process. One deduces that Theorem \ref{all} applies, and one can construct the 
corresponding measure $\mathcal{Q}$. Let us check that we are in the situation where $L_{\infty} = \infty$, 
$\mathbb{P}$-almost surely. Indeed, let $T^{(0)}$ be the first hitting time of zero 
by $(X_s)_{s \geq 0}$, and for $n \geq 1$, let $T^{(n)}$ be the first hitting time of zero after $T^{(n-1)} + 1$. 
The variables $(T^{(n)})_{n \geq 0}$ are stopping times, $\mathbb{P}$-almost surely finite since 
$(X_s)_{s \geq 0}$
is recurrent. By the strong Markov property the variables
 $(L_{T^{(n)}} - L_{T^{(n-1)}})_{n \geq 1}$ are i.i.d., which implies that $L_{\infty} = 0$ almost surely,
 or $L_{\infty}= \infty$ almost surely. Let us suppose that the first case holds. One deduces 
that $(X_s)_{s \geq 0}$ is a martingale, and by the optional sampling 
theorem, for all $s, u \geq 0$, $$\mathbb{P} [X_{(T^{(0)} \wedge u) + s} |\mathcal{F}_{T^{(0)} \wedge u} ] 
= X_{T^{(0)} \wedge u},$$
which implies:
$$\mathbb{P}[X_{T^{(0)} + s} \mathds{1}_{T^{(0)} \leq u} ] = \mathbb{P} [X_{T^{(0)}} \mathds{1}_{T^{(0)}
 \leq u}] = 0,$$
and  then $X_{T^{(0)} + s} = 0$ a.s. on the event $\{T^{(0)} \leq u\}$. Since $u$ can be arbitrarily chosen and 
$T^{(0)} < \infty$ a.s.,  $X_{T^{(0)} + s} = 0$ a.s., which contradicts the fact that zero is an
instantaneously reflecting barrier. The fact that $L_{\infty} = \infty$ implies that $g^{[a]} < \infty$
$\mathcal{Q}$-almost everywhere, for
$$g^{[a]} := \sup\{s \geq 0, X_s \leq a \},$$
as proved in \cite{NN3}. This property is an important element 
 in the proof of our penalisation result, which can now be rigorously stated.
The relevant class of functionals $(\Gamma_t)_{t \geq 0}$ is defined as follows:
 \begin{definition}
Let us suppose that the assumptions of Theorem \ref{all} are satisfied. We say that 
a process $(\Gamma_t)_{t \geq 0}$ belongs to the class (C) if it is nonnegative, uniformly bounded,
nonincreasing, c\`adl\`ag and adapted with 
respect to $(\mathcal{F}_t)_{t \geq 0}$, if there exists $a > 0$ such that for all 
$t \geq 0$, $\Gamma_t = \Gamma_{g^{[a]}}$ on the set $\{t \geq g^{[a]}\}$,
and if the decreasing limit of $\Gamma_t$ at infinity, denoted $\Gamma_{\infty}$, is $\mathcal{Q}$-integrable. 
\end{definition}
\noindent
This definition is stated in \cite{NN3}, and in a slightly different way in \cite{NRY}. 
The main result of this article is the following:
\begin{theorem} \label{penalisation}
We suppose that the filtered probability space $(\Omega, \mathcal{F}, (\mathcal{F}_s)_{s \geq 0}, 
\mathbb{P})$ and the diffusion process $X$ are constructed as above. Let $\mathcal{Q}$ be the $\sigma$-finite measure associated with $X$ from Theorem \ref{all} and let $(\Gamma_t)_{t \geq 0}$ be a process in the class (C),
such that $\mathcal{Q} [\Gamma_{\infty}] > 0$. For all $t \geq 0$:
$$0< \mathbb{P} [\Gamma_t] < \infty,$$
and one can  define a probability measure $\mathbb{Q}_t$ on $(\Omega, \mathcal{F})$ by
$$\mathbb{Q}_t := \frac{\Gamma_t}{\mathbb{P}[\Gamma_t]} \, . \mathbb{P}.$$
Then the probability measure:
$$\mathbb{Q}_{\infty} := \frac{\Gamma_{\infty}}{ \mathcal{Q} [\Gamma_{\infty}]} \, . \mathcal{Q}$$ 
is the weak limit of $\mathbb{Q}_t$ in the sense of penalisations,  i.e. for all $s \geq 0$, 
and for all events $\Lambda_s \in \mathcal{F}_s$,
$$\mathbb{Q}_t [\Lambda_s] \underset{t \rightarrow \infty}{\longrightarrow} \mathbb{Q}_{\infty} [\Lambda_s].$$
\end{theorem}
\noindent
The proof of Theorem \ref{penalisation} is given in Section \ref{sturm} and Section \ref{proof}. 
\begin{remark}
This penalisation result applies in particular for the power $2r$
of a Bessel process of dimension $2(1-r)$, if $0 < r < 1$. 
For $r=1$, the corresponding diffusion is
 a reflected Brownian motion, and our result is very similar to the result obtained for Brownian motion in 
\cite{NRY}. The proof of Theorem \ref{penalisation} is given in Sections \ref{sturm} and \ref{proof}.
\end{remark}
\section{Estimates of hitting times} \label{sturm} 
In order to prove Theorem \ref{penalisation}, we need to estimate the distribution of the
 hitting times of the canonical process $(X_s)_{s \geq 0}$ under $\mathbb{P}$. These estimates 
involve Sturm-Liouville equation in a crucial way, that is why we first prove the following lemma,  
giving some information about the solutions of this equation:  
\begin{lemma} \label{sturmliouville}
 Let $A$ be a continuous function from $\mathbb{R}_+^*$ to $\mathbb{R}_+^*$, integrable in the neighborhood 
of zero, but not in the neighborhood of infinity. Then there exists a unique function $\Phi_A$ from 
$\mathbb{R}_+$ to $\mathbb{R}$, continuous on $\mathbb{R}_+$, bounded, twice differentiable on $\mathbb{R}_+^*$,
such that $\Phi_A (0) = 1$ and satisfying Sturm-Liouville equation:
\begin{equation}
\Phi_A''(t) = A(t) \Phi_A(t) \label{SL}
\end{equation}
for all $t > 0$. This function is strictly positive, decreasing to zero at infinity, and
 continuously differentiable everywhere in $\mathbb{R}_+$. Moreover, if $A_1$ and $A_2$ are two functions
satisfying the same assumptions as $A$, and if $A_1(t) \geq A_2(t)$ for all $t > 0$, then 
$|\Phi_{A_1}'(0)| \geq |\Phi_{A_2}'(0)|$.
\end{lemma}
\begin{proof}
By Cauchy-Lipschitz theorem, for all $t_0 > 0$, $a, b \in \mathbb{R}$, there exists a unique maximal solution 
$\Theta_{t_0, a, b}$ of \eqref{SL} on an interval of the form $[t_0, v)$ for some $v \in (t_0, \infty]$,  
which satisfies $\Theta_{t_0,a,b}(t_0) = a$ and $\Theta'_{t_0,a,b}(t_0) = b$. Since $A$ is locally bounded
on $\mathbb{R}_+^*$, $\Theta_{t_0,a,b}$ is in fact well-defined on $[t_0, \infty)$. 
By linearity, one has:
$$\Theta_{t_0,a,b} = a \Theta_{t_0,1,0} + b \Theta_{t_0, 0, 1}.$$
For all $t > t_0$, $\Theta_{t_0,1,0} (t)$, $\Theta'_{t_0,1,0} (t)$,
$\Theta_{t_0,0,1} (t)$ and $\Theta'_{t_0,0,1} (t)$ are strictly positive, moreover, 
 $\Theta_{t_0,1,0}$ and $\Theta_{t_0,0,1}$ tend to infinity at infinity. 
On the other hand, let us suppose that $\Theta_{t_0,a,b}$ is not bounded from above (resp. from below). 
Then there exists $t > t_0$ such that $\Theta_{t_0,a,b}(t) > 0$ and $\Theta'_{t_0,a,b}(t) \geq 0$ 
(resp. $\Theta_{t_0,a,b}(t) < 0$ and $\Theta'_{t_0,a,b}(t) \leq 0$): one deduces
 that $\Theta_{t_0,a,b}$ tends to infinity (resp. minus infinity) at infinity. Hence, for $b_1 > b_2$, 
at least one of the two following  holds:
\begin{itemize}
\item $\Theta_{t_0,a,b_1}$ tends to infinity at infinity.
\item $\Theta_{t_0,a,b_2}$ tends to minus infinity at infinity.
\end{itemize}
\noindent
Indeed, if none of these  holds, then $\Theta_{t_0,a,b_1}$ is bounded from above, and 
$\Theta_{t_0,a,b_2}$ is bounded from below, which implies that $\Theta_{t_0,0,1}$ is bounded from above: a 
contradiction. Now, for $b$ large enough, $\Theta_{t_0,a,b} (t_0 +1)$ and $\Theta'_{t_0,a,b} (t_0 +1)$
are strictly positive, which implies that $\Theta_{t_0,a,b}$ tends to infinity at infinity, similarly, 
for $b$ small enough, $\Theta_{t_0,a,b}$ tends to minus infinity at infinity. One deduces that there 
exists $b(t_0,a)$ such that for $b > b(t_0,a)$, $\Theta_{t_0,a,b}$ tends to infinity and 
for $b < b(t_0,a)$, $\Theta_{t_0,a,b}$ tends to minus infinity. Let us suppose 
that for some $t \geq t_0$, $\Theta_{t_0,a,b(t_0,a)}(t)$ and $\Theta'_{t_0,a,b(t_0,a)}(t)$ have the 
same sign, and that at least one of them is different from zero.  Since one can increase $t$ by a small 
quantity, one can suppose that $\Theta_{t_0,a,b(t_0,a)}(t)$ and $\Theta'_{t_0,a,b(t_0,a)}(t)$ have strictly 
the same sign (recall that $\Theta_{t_0,a,b(t_0,a)} (t) > 0$ implies that $\Theta''_{t_0,a,b(t_0,a)} (t) > 0$).
On deduces that for $b$ sufficiently close to $b(t_0,a)$,
 $\Theta_{t_0,a,b}(t)$ and $\Theta'_{t_0,a,b}(t)$ have the same sign, independent of $b$. Therefore,
$\Theta_{t_0,a,b}$ and $\Theta'_{t_0,a,b}$ tend to a limit equal to infinity or minus infinity, 
independently of $b$: this is a contradiction. At this step, we know that for all $t \geq t_0$,
$\Theta_{t_0,a,b(t_0,a)}(t) = \Theta'_{t_0,a,b(t_0,a)} (t) = 0 $ 
or $\Theta_{t_0,a,b(t_0,a)}(t) \Theta'_{t_0,a,b(t_0,a)} (t) < 0 $.
If the first case holds for some $t \geq t_0$, by using Cauchy-Lipschitz theorem in both directions of the 
time, we deduce that $a= b(t_0,a) = 0$ and $\Theta_{t_0,a,b(t_0,a)}$ is identically zero.
Otherwise, $\Theta^2_{t_0,a,b(t_0,a)}$ is strictly decreasing, which implies that it remains 
strictly positive, hence, $\Theta_{t_0,a,b(t_0,a)}$ is strictly positive and strictly decreasing for $a> 0$,
strictly negative and strictly incrasing for $a < 0$. In particular, it converges almost surely to a limit $l$.
If $l> 0$, then for all $t \geq t_0$, $\Theta_{t_0,a,b(t_0,a)}(t)
\geq l$ and $\Theta''_{t_0,a,b(t_0,a)}(t) \geq lA(t)$, which implies that 
$\Theta'_{t_0,a,b(t_0,a)}$ tends to infinity at infinity, since  $A$ is not integrable at infinity: this is
impossible. Since $l< 0$ gives also a contradiction, one has $l=0$. 
To summarize, we have proved that for all $t_0 > 0$ and all $a \in \mathbb{R}$, there exists 
a unique solution $\Theta_{t_0,a}$ of \eqref{SL} defined on $[t_0, \infty)$, equal to $a$ at $t_0$, and
which does not tend to infinity or minus infinity at infinity. This solution is identically 
zero if $a = 0$, it is strictly monotone and tends to zero at infinity, if $a \neq 0$. 
In fact, by linearity, one has $\Theta_{t_0,a} = a \Theta_{t_0,1}$ for all $a \in \mathbb{R}$.
The uniqueness of $\Theta_{t_0,a}$ implies that for $t_0 > t_1 > 0$, $a \in \mathbb{R}$, the restriction 
of $\Theta_{t_1,a}$ to $[t_0, \infty)$ is equal to $\Theta_{t_0, b}$ where 
$b = \Theta_{t_1,a} (t_0)$. This compatibility implies that if we define $\Theta$ from $\mathbb{R}_+^*$ 
to $\mathbb{R}_+^*$ by:
$$\Theta(t) = \Theta_{1,1} (t)$$
for $t \geq 1$, and 
$$\Theta(t) = \frac{1}{\Theta_{t,1}(1)}$$
for $t < 1$, then $\Theta$ is a solution of \eqref{SL}, strictly decreasing and tending to zero at 
infinity. 
Now, since $|\Theta'|$ is decreasing, one has, for all $t \in (0,1)$,
$$\Theta (t) \leq \Theta(1) + \int_{t}^1 |\Theta'(s)| \, ds \leq 1 + |\Theta'(t)| \leq L |\Theta'(t)|$$
for $L := 1 + 1/|\Theta'(1)|$. One deduces that
$$\Theta''(t) \leq  LA(t) |\Theta'(t)|,$$
and then 
$$\left| \,\frac{d}{dt} \, \log |\Theta'(t)| \, \right| \leq L A(t).$$
Since $A$ is integrable in the neighborhood of zero, $|\Theta'(t)|$, and then $\Theta(t)$, is
bounded for $t \in (0,1)$. Since $\Theta$ is decreasing, it can be defined at zero by continuity. Moreover,
 $\Theta'(t)$ converges when $t$ goes to zero, and 
one can check that its limit is equal to the derivative at zero of the extension of $\Theta$. One then 
deduces the existence and the properties of regularity of $\Phi_A$ by setting, for all $t \geq 0$:
$$\Phi_A(t) = \frac{\Theta(t)}{\Theta(0)}.$$
To prove uniqueness, let us suppose that $\Phi_1$ and $\Phi_2$ satisfy the conditions given in Lemma
\ref{sturmliouville}. The functions $\Phi_1$ and $\Phi_2$ cannot vanish on $\mathbb{R}_+^*$: otherwise,
 by the properties of the solutions
of \eqref{SL} on the intervals of the form $[t_0, \infty)$ for $t_0 > 0$, they would be identically zero. 
Then $\Phi_1$, $\Phi_2$ are strictly positive, and they are proportional to each other on every interval
of the form $[t_0, \infty)$ for $t_0 > 0$. In other words, $\Phi_1 / \Phi_2$ is constant on 
$[t_0, \infty)$ for all $t_0 > 0$,  and then on $\mathbb{R}^*$, finally, on $\mathbb{R}_+$ by 
continuity at zero. Since $\Phi_1 (0) =\Phi_2(0)= 1$, $\Phi_1 = \Phi_2$ everywhere. 
Let us now suppose that $A_1 \geq A_2$ satisfy the assumptions of Proposition \ref{SL}, 
and that $|\Phi'_{A_1}(0)| < |\Phi'_{A_2} (0)|$. One has $\Phi'_{A_1}(0) > \Phi'_{A_2} (0)$. 
Let $T> 0$ be the infimum of the times $t$ such that 
$\Phi'_{A_1}(t) \leq \Phi'_{A_2} (t)$. On the interval $(0, T)$, one has
$\Phi'_{A_1} > \Phi'_{A_2}$, 
 and then $\Phi_{A_1} \geq \Phi_{A_2}$, since $\Phi_{A_1} (0) = \Phi_{A_2}(0)$.
Now, since $A_1 \geq A_2$, one deduces that $$\Phi''_{A_1} = A_1 \Phi_{A_1} \geq A_2 \Phi_{A_2} = \Phi''_{A_2},$$
which implies: 
\begin{equation} \Phi'_{A_1} \geq  \Phi'_{A_2} + \Phi'_{A_1}(0) - \Phi'_{A_2} (0). \label{j36}
\end{equation}
If $T$ is supposed to be finite, one deduces a contradiction at time $T$. Hence $T$ is infinite and
\ref{j36} holds everywhere. One deduces that 
$$\Phi_{A_1}(t) - \Phi_{A_2}(t) \geq t \left( \Phi'_{A_1}(0) - \Phi'_{A_2} (0) \right)$$ for all $t \geq 0$, 
and then $\Phi_{A_1} - \Phi_{A_2}$ tends to infinity at infinity, which is absurd. 
\end{proof}
\noindent
Once Lemma \ref{SL} is proved, one can state and show the results on hitting times which are involved 
in the proof of Theorem \ref{penalisation}. In the sequel of the paper, for any $x \geq 0$, 
 we denote by
 $\mathbb{P}_x$ the distribution of a diffusion which starts at $x$, with
  the same parameters as the canonical process under $\mathbb{P}$. In particular, 
$\mathbb{P}_{x_0}  = \mathbb{P}$. With this notiation, one has the following lemma:
\begin{lemma} \label{exp}
For all $\lambda > 0$, there exists a unique function $\Phi_{\lambda}$ such that $\Phi_{\lambda}(0) = 1$
and for all $x \geq y \geq 0$:
$$\mathbb{P}_{x} [e^{-\lambda T_y}] = \frac{\Phi_{\lambda} (x \lambda^{\alpha})}
 {\Phi_{\lambda} (y \lambda^{\alpha})},$$
where $T_y$ is the first hitting time of $y$ by the canonical process, and $\alpha:= 1/(\beta+ 2)$ lies in 
the interval $(0,1)$. Moreover, there exists $\lambda_{0} > 0$ such that the following properties hold:
\begin{itemize}
\item $\Phi_{\lambda}(t)$ converges to one when $(\lambda, t)$ goes to zero. 
\item $\Phi_{\lambda} (t)$ converges to zero when $t$ goes to infinity, uniformly in $\lambda \in (0, \lambda_0)$.
\item For all $\lambda > 0$, $\Phi_{\lambda}$ is continuously differentiable everywhere in $\mathbb{R}_+$.
\item  The map: $(\lambda, t) \longrightarrow \Phi'_{\lambda}(t)$ is 
unformly bounded on $(0, \lambda_0) \times \mathbb{R}_+$. 
\item $|\Phi'_{\lambda} (t)|$ converges to a constant $K > 0$ when $(\lambda, t)$ goes to zero. 
\end{itemize}
\end{lemma}
\noindent
\begin{proof}
Let us define, for all $u \geq 0$:
$$\Psi_{\lambda} (u) := \mathbb{P}_{u} [e^{-\lambda T_0}].$$
By the strong Markov property one has
$$\mathbb{P}_x [e^{-\lambda T_0}] = \mathbb{P}_x [e^{-\lambda T_y}] \, \mathbb{P}_y [e^{-\lambda T_0}],$$
and then 
$$\mathbb{P}_x [e^{-\lambda T_y}] =  \frac{\Phi_{\lambda} (x \lambda^{\alpha})}
 {\Phi_{\lambda} (y \lambda^{\alpha})},$$
where for $t \geq 0$,
$$\Phi_{\lambda} (t) := \Psi_{\lambda} (t \lambda^{-\alpha}).$$
 One deduces the existence of  $\Phi_{\lambda}$, its uniqueness is clear. 
Now, by classical properties of diffusions (see \cite{BS} for example), the functions $\Psi_{\lambda}$, and
then $\Phi_{\lambda}$, are 
continuous on $\mathbb{R}_+$, twice differentiable on $\mathbb{R}_+^*$ and $\Psi_{\lambda}$ satisfies:
$$\mathcal{G} \Psi_{\lambda} = \lambda \, \Psi_{\lambda} $$
where $\mathcal{G}$ is the infinitesimal generator of the diffusion $X$. One deduces:
$$\Phi''_{\lambda} (t) = t^{\beta}  p(t \lambda^{-\alpha}) \, \Phi_{\lambda}(t),$$
for 
$$p(u) := u^{-\beta} m(u).$$
By assumption, $p(u)$ is strictly positive, continuous with respect to $u > 0$, it tends to $c$ when $u$
 goes to infinity, and for all $u >0$,
 $$p(u) \leq C \left(1+ u^{-\beta} \mathds{1}_{\beta > 0}\right).$$
One deduces that $$A_{\lambda} : t \longrightarrow t^{\beta} p(t \lambda^{-\alpha})$$
satisfies the assumptions of Lemma \ref{sturmliouville}, moreover, with the notation of 
this Lemma: $$\Phi_{\lambda} = \Phi_{A_{\lambda}},$$
in particular, this function is continuously differentiable everywhere in $\mathbb{R}_+$ (third item of Lemma \ref{exp}).
Now, let us define the functions:
$$A_0 : t \longrightarrow c \, t^{\beta},$$
and for $\epsilon > 0$, $\lambda \geq 0$: 
$$A_{\lambda, \epsilon} : t \longrightarrow A_{\lambda} (t+\epsilon).$$
Since $p$ tends to $c$ at infinity, for all $\epsilon> 0$, there exists $\lambda(\epsilon) \in (0,1)$ such that 
for all $\lambda \in [0, \lambda(\epsilon)]$:
$$ (1- \epsilon) A_{0,\epsilon} \leq A_{\lambda, \epsilon} \leq (1+\epsilon)A_{0, \epsilon}.$$ 
All these functions satisfy the conditions of Lemma \ref{sturmliouville}, and one deduces: 
$$ \left|\Phi'_{(1-\epsilon) A_{0,\epsilon}} (0) \right| \leq \left| \Phi'_{A_{\lambda, \epsilon} } (0) \right| 
\leq \left|\Phi'_{(1+\epsilon) A_{0,\epsilon}} (0) \right| $$
Now, by checking Sturm-Liouville equation, one sees that for all $t \geq 0$:
$$\Phi_{A_{\lambda, \epsilon}} (t) = \frac{ \Phi_{A_{\lambda}} (t+ \epsilon)}{\Phi_{A_{\lambda}} (\epsilon)},$$
$$\Phi_{(1-\epsilon) A_{0,\epsilon}} (t) = \frac{\Phi_{A_0} \left[(1-\epsilon)^{\alpha} (t+\epsilon) \right] }
{\Phi_{A_0} \left[(1-\epsilon)^{\alpha} (\epsilon) \right] }$$
and 
$$\Phi_{(1+\epsilon) A_{0,\epsilon}} (t) = \frac{\Phi_{A_0} \left[(1+\epsilon)^{\alpha} (t+\epsilon) \right] }
{\Phi_{A_0} \left[(1+\epsilon)^{\alpha} (\epsilon) \right] }.$$
Therefore:
$$ K_1(\epsilon) \leq \frac{\left| \Phi'_{A_{\lambda} } (\epsilon) \right| }{ \Phi_{A_{\lambda}} (\epsilon)}
\leq K_2(\epsilon),$$
where 
$$K_1(\epsilon) := (1-\epsilon)^{\alpha} \, 
\frac{\left|\Phi'_{A_0} \left[(1-\epsilon)^{\alpha} (\epsilon) \right]\right| }
{\Phi_{A_0} \left[(1-\epsilon)^{\alpha} (\epsilon) \right] }$$
and 
$$K_2(\epsilon) := (1+\epsilon)^{\alpha} \, \frac{\left|\Phi'_{A_0} \left[(1+\epsilon)^{\alpha} (\epsilon) \right]\right| }
{\Phi_{A_0} \left[(1+\epsilon)^{\alpha} (\epsilon) \right] }$$
do not depend on $\lambda \leq \lambda(\epsilon)$ and tend to $K := |\Phi'_{A_0} (0)|$ 
when $\epsilon$ goes to zero. One deduces that
$$|\Phi'_{A_{\lambda}} (\epsilon) | \leq K_2(\epsilon)$$
and 
$$|\Phi'_{A_{\lambda}} (0) | \leq K_2(\epsilon) + \int_0^{\epsilon} A_{\lambda}(u) \Phi_{A_{\lambda}} (u) \, du
\leq  K_2 (\epsilon) + C \int_0^{\epsilon} u^{\beta} \left[ 1 + (u \lambda^{-\alpha})^{-\beta} \mathds{1}_{\beta > 0} \right] \, du.$$
Since one supposes $\lambda \leq \lambda(\epsilon) \leq 1$, one deduces that for all $t \geq 0$:
$$|\Phi'_{A_{\lambda}} (t) | \leq |\Phi'_{A_{\lambda}} (0) | \leq K_2(\epsilon) + I(\epsilon),$$
where 
$$I(\epsilon) := C \int_0^{\epsilon} (1+ u^{\beta}) \, du<\infty$$
decreases to zero when $\epsilon$ goes to zero. 
In other words:
$$|\Phi'_{A_{\lambda}} (t) | \leq K_3(\epsilon),$$
where $K_3(\epsilon)< \infty$ decreases to $K$ when $\epsilon$ goes to zero. 
In particular, for $\lambda \leq \lambda(1)$:
$$|\Phi'_{A_{\lambda}} (t) | \leq K_3(1),$$
which proves the fourth item of Lemma \ref{exp}.
Moreover, for all $\epsilon > 0$:
$$\underset{(\lambda,t) \rightarrow 0}{\lim \, \sup} \, |\Phi'_{A_{\lambda}} (t) |
\leq \underset{\lambda \in [0, \lambda(\epsilon)), t \in \mathbb{R}_+}{\sup} |\Phi'_{A_{\lambda}} (t) |
\leq K_3(\epsilon),$$
and then, by taking $\epsilon$ going to zero, 
\begin{equation} 
\underset{(\lambda,t) \rightarrow 0}{\lim \, \sup} |\Phi'_{A_{\lambda}} (t) | \leq K. \label{limsupphi}
\end{equation}
\noindent
On the other hand, again for $\lambda \in [0, \lambda(\epsilon)]$, 
one has
$$\Phi_{A_{\lambda}}(\epsilon) \geq 1 - \int_0^{\epsilon} |\Phi'_{A_{\lambda}} (u)| \, du
\geq 1 - \epsilon K_3(\epsilon),$$
and then
$$|\Phi'_{A_{\lambda}}(\epsilon)| \geq [1-\epsilon K_3(\epsilon)]_+ \frac{|\Phi'_{A_{\lambda}}(\epsilon)|}
{\Phi_{A_{\lambda}}(\epsilon)} \geq  [1-\epsilon K_3(\epsilon)]_+ K_1(\epsilon) =: K_4(\epsilon),$$
where $K_4(\epsilon)$ tends to $K$ when $\epsilon$ tends to zero. 
We deduce that
$$|\Phi'_{A_{\lambda}}(t)| \geq K_4(\epsilon)$$
for all $t \leq \epsilon$, and then
$$\underset{(\lambda,t) \rightarrow 0}{\lim \, \inf} |\Phi'_{A_{\lambda}} (t) |
\geq \underset{\lambda \in [0, \lambda(\epsilon)), t \in [0, \epsilon]}{\inf} |\Phi'_{A_{\lambda}} (t) |
\geq K_4(\epsilon).$$
This implies, by letting $\epsilon$ go to zero that
\begin{equation} 
\underset{(\lambda,t) \rightarrow 0}{\lim \, \inf} |\Phi'_{A_{\lambda}} (t) | \geq K. \label{liminfphi}
\end{equation}
The inequalities \eqref{limsupphi} and \eqref{liminfphi} imply the fifth item of Lemma \ref{exp}. 
Moreover, for all $\lambda \geq 0$ and $t \geq 0$:
$$1 - t |\Phi'_{A_{\lambda}} (0)| \leq \Phi_{A_{\lambda}} (t) \leq 1.$$
Hence, for $\lambda \leq \lambda(1)$:
$$1 - t K_3(1) \leq \Phi_{A_{\lambda}} (t) \leq 1.$$
Since $tK_3(1)$ tends to zero when $(\lambda,t) \in [0,\lambda(1)] \times \mathbb{R}_+$
 tends to zero, one deduces the first item of Lemma \ref{exp}. It now remains to prove the second item. 
Let us suppose that $\lambda \leq 1$ and $t \geq 1$. In this case:
$$A_{\lambda,t} \geq \delta A_{0,t},$$
where $$ \delta := \frac{1}{c} \, \underset{v \in [1,\infty)}{\inf} p(v) > 0.$$
One deduces that
$$\frac{|\Phi'_{A_{\lambda}}(t)|}{\Phi_{A_{\lambda}} (t)} \geq \delta^{\alpha} \, 
\frac{|\Phi'_{A_{0}}(\delta^{\alpha} t)|}{\Phi_{A_{0}} (\delta^{\alpha} t)},$$
or 
$$\frac{d}{dt}  \left[\log \left( \Phi_{A_{\lambda}}(t) \right)  \right]
\leq \frac{d}{dt}  \left[\log \left( \Phi_{A_{0}}(\delta^{\alpha}t) \right)  \right].$$
By integrating and taking the exponential, one obtains
$$\Phi_{A_{\lambda}} (t) \leq \Phi_{A_{\lambda}} (1) \, \frac{\Phi_{A_{0}}(\delta^{\alpha}t) }
{\Phi_{A_{0}}(\delta^{\alpha}) },$$
which implies, for all $\lambda \leq 1$ and $t \geq 1$:
$$\Phi_{A_{\lambda}} (t) \leq \theta(t),$$
where
$$\theta(t) =  \frac{\Phi_{A_{0}}(\delta^{\alpha}t) }
{\Phi_{A_{0}}(\delta^{\alpha}) }.$$
We can remark that this upper bound trivially holds for $t < 1$, since $\theta(t)$
 is greater than or equal to one in this case. 
Now, $\theta (t)$ tends to zero when $t$ goes to infinity (since $\Phi_{A_0}$ tends to zero), independently
of $\lambda \leq 1$, which completes the proof of Lemma \ref{exp}. 
\end{proof}
\noindent
The following lemma gives some information about the marginal distributions of the canonical process
 under $\mathbb{P}_x$:
\begin{lemma} \label{infu}
One can define a function $\chi$ from $\mathbb{R}_+^*$ to $\mathbb{R}_+$, tending to 
zero at zero, and satisfying the following property: for all $u > 0$, 
there exists $t_0(u) > 0$ such that for all $x \geq 0$, $t \geq  t_0(u)$:
$$\mathbb{P}_x [X_t \leq t^{\alpha} u] \leq \chi(u).$$
\end{lemma}
\begin{proof}
Let us fix $u > 0$. By coupling, one obtains immediately, for all $x \geq 0$, $t > 0$,
$$\mathbb{P}_x [X_t \leq t^{\alpha} u] \leq \mathbb{P}_0 [X_t \leq t^{\alpha} u].$$
Now one has (see, for example, \cite{NRY}, p. 89):
$$\mathbb{P}_0 [X_t \leq t^{\alpha} u] = \int_{0}^{t^{\alpha}u} p(t,0,y) \, m(y) \, dy
=  \int_{0}^{t^{\alpha}u} p(t,y,0) \, m(y) \, dy,$$
where $p$ is the density of the semi-group of $(X_t)_{t \geq 0}$. Again by coupling, one sees that
 $p(t,y,0) \leq p(t,0,0)$, which implies that 
$$\mathbb{P}_x [X_t \leq t^{\alpha} u] \leq p(t,0,0) \int_{0}^{t^{\alpha}u} C \left(y^{\beta} + 
\mathds{1}_{\beta > 0} \right) \, dy \leq C \, p(t,0,0) \, \left[ \left(t^{\alpha}u \right)^{\beta+1}/(\beta + 1)
+ t^{\alpha}u  \mathds{1}_{\beta> 0}\right].$$
Now, if $t \geq u^{-1/\alpha}$, one has necessarily $t^{\alpha}u \leq (t^{\alpha}u)^{\beta+1}$ for
$\beta > 0$, which implies:
$$\mathbb{P}_x [X_t \leq t^{\alpha} u] \leq \frac{C (\beta+2)}{\beta+1} \, p(t,0,0)
 \, \left(t^{\alpha}u \right)^{\beta+1} = \chi_0(u) \, t^{1-\alpha} p(t,0,0) $$
where
$$\chi_0(u) := \frac{C (\beta+2)}{\beta+1} u^{\beta+1}.$$
Hence, we are done, provided that there exists $L_1 > 0$ such that $p(t,0,0) \leq L_1 \, t^{\alpha-1}$ for
$t$ large enough. If $t > t' > 0$, one has
$$p(t,0,0) = \mathbb{P}_{0} [p(t',X_{t-t'},0)] \leq p(t',0,0),$$
which implies that $p(t,0,0)$ is decreasing with respect to $t$. One deduces:
$$p(t,0,0) \leq \frac{2e}{t} \int_{t/2}^t e^{-s/t} p(s,0,0) \, ds \leq \frac{2e}{t} \, R(1/t),$$
where for $\lambda > 0$,
$$R(\lambda) := \int_0^{\infty} e^{-s\lambda} p(s,0,0) \, ds.$$
Now from a result  by Salminen, Vallois and Yor in \cite{SVY} (p. 5, equation just after (3)), one obtains
$$\int_0^{\infty} m(y) \mathbb{P}_y [e^{-\lambda T_0}] \, dy = \frac{1}{\lambda R(\lambda)},$$
which implies, with the notation of Lemma \ref{exp}, that
$$p(t,0,0) \leq 2e \, \left( \int_0^{\infty} m(y) \, \Phi_{1/t}(y t^{-\alpha}) \, dy \right)^{-1}.$$
Therefore we only need to prove that for some $L_2 > 0$, and for $t$ large enough
$$I := \int_0^{\infty} m(y) \, \Phi_{1/t}(y t^{-\alpha}) \, dy \geq L_2 t^{1-\alpha}.$$
Now, from Lemma \ref{exp} (first item), we know that there exist $t_0, v_0 > 0 $
such that for all $t \geq t_0$ and all $v \leq v_0$, $\Phi_{1/t}(v) \geq 1/2$. 
One deduces that for all $t \geq t_0 \vee v_0^{-1/\alpha}$, $v_0 t^{\alpha} \geq 1$:
$$ I \geq \frac{1}{2} \int_{1/2}^{v_0 t^{\alpha}} m(y) dy \geq 
\delta \frac{(v_0 t^{\alpha})^{\beta+1} -(1/2)^{\beta+1}} {2(\beta+1)}
\geq \frac{\delta v_0^{\beta+1} (1-(1/2)^{\beta+1})}{2(\beta +1)} t^{1-\alpha}.$$
where $\delta$ is the infimum of $y^{-\beta} m(y)$ on $[1/2,\infty)$, strictly positive
since $m(y)$ is strictly positive, continuous and equivalent to $c y^{\beta}$ at infinity. 
\end{proof}
\noindent
From Lemma \ref{infu}, one deduces the following result, which majorizes the probability, 
for the canonical process, to hit $[0,a]$ during a small interval of time:
\begin{lemma} \label{gamma}
One can define a function $\rho$ from $(0,1)$ to $\mathbb{R}_+^*$, tending to zero at zero, and 
satifying the following property: for all $a \geq 0$, $\gamma \in (0,1)$, there exists $t_0(a,\gamma)$ such that 
for every $x \geq 0$, $t \geq t_0(a,\gamma)$:
$$\mathbb{P}_x \left[ \exists s \in [(1-\gamma) t, t], X_s \leq a \right] \leq \rho (\gamma).$$
\end{lemma}
\begin{proof}
Let us fix $a$ and $\gamma$. 
One has immediately:
$$\mathbb{P}_x \left[ \exists s \in [(1-\gamma) t, t], X_s \leq a \right] =
 \mathbb{P}_x [\Delta(X_{(1-\gamma) t})],$$
where, for $y \leq a$:
$$\Delta (y) = 1$$
and for $y > a$:
$$\Delta (y) = \mathbb{P}_y [ T_a \leq \gamma t ].$$
In any case, one has:
$$\Delta(y) \leq e \, \mathbb{P}_{y \vee a} [e^{-T_a/ (\gamma t)}]  \leq e \, \frac{\Phi_{1/(\gamma t)} 
[y (\gamma t)^{-\alpha}] }{\Phi_{1/(\gamma t)} 
[a (\gamma t)^{-\alpha}] }.$$
Now, there exist $\lambda_1, v_1 > 0$ such that $\Phi_{\lambda} (v) \geq e/3$ for all $\lambda \leq \lambda_1$ 
and $v \leq v_1$ (first item of Lemma \ref{exp}). Hence, there exists $t_1 > 0$ (depending on
 $\gamma$ and $a$) such that for $t \geq t_1$:
$$\Phi_{1/(\gamma t)} [a (\gamma t)^{-\alpha}] \geq e/3$$
and then for all $y \geq 0$:
$$\Delta(y) \leq 3 \, \Phi_{1/(\gamma t)} [y (\gamma t)^{-\alpha}]$$
By the second item of Lemma \ref{exp}, there exists a function $q$, bounded by one, decreasing to zero
 at infinity, such that for all $\lambda \leq \lambda_0$, $v > 0$:
$$\Phi_{\lambda} (v) \leq q(v).$$
Hence, if  $t \geq t_1 \vee 1/(\lambda_0 \gamma)$, for all $y \geq 0$,
$$\Delta(y) \leq 3 \, q \left[ y (\gamma t)^{-\alpha} \right],$$
which implies:
$$\mathbb{P}_x \left[ \exists s \in [(1-\gamma) t, t], X_s \leq a \right] \leq
3 \, \mathbb{P}_x \left[q \left( X_{(1-\gamma) t} (\gamma t)^{-\alpha} \right)    \right],$$
and then, for any $v \geq 0$:
$$\mathbb{P}_x \left[ \exists s \in [(1-\gamma) t, t], X_s \leq a \right]
\leq 3 \left( q(v) + \mathbb{P}_x \left[ X_{(1-\gamma) t} \leq (\gamma t)^{\alpha}v \right] \right).$$
We now fix 
$$v := \left (\frac{\gamma}{1-\gamma} \right)^{-\alpha/2}.$$ 
One deduces, by Lemma \ref{infu}, that there exists $t_2 \geq 0$ (depending on $a$ and $\gamma$), such 
that for $t \geq t_2$:
$$\mathbb{P}_x \left[ \exists s \in [(1-\gamma) t, t], X_s \leq a \right]
\leq 3 \left( q\left[\left (\frac{\gamma}{1-\gamma} \right)^{-\alpha/2} \right] 
+ \chi \left[\left (\frac{\gamma}{1-\gamma} \right)^{\alpha/2} \right] \right) =: \rho(\gamma),$$
which tends to zero at zero. 
\end{proof}
\noindent
We have now all the estimates of hitting times needed in the proof 
of Theorem \ref{penalisation}, which is finished in Section \ref{proof}. 
\section{Proof of the main theorem} \label{proof}
In order to prove Theorem \ref{penalisation}, we essentially need to estimate the behaviour of
the expectation of $\Gamma_t$ under $\mathbb{P}$, when $t$ goes to infinity. This expectation
will be implicitly splitted as follows:
$$\mathbb{P} [\Gamma_t] = \mathbb{P}[\Gamma_t \mathds{1}_{g_t^{[a]} \leq s}] 
+ \mathbb{P} [\Gamma_t \mathds{1}_{s < g_t^{[a]} \leq (1- \gamma)t} ]
+  \mathbb{P} [\Gamma_t \mathds{1}_{ g_t^{[a]} > (1- \gamma)t} ],$$
where $s \geq 0$, $\gamma \in (0,1/2)$ and
$$g_t^{[a]} := \sup \{u \in [0,t], X_u \leq a\}$$
(recall that $g^{[a]}$ denotes the supremum of $u \in \mathbb{R}_+$ such that $X_u \leq a$).
Moreover, we shall use a Tauberian theorem: this is the reason why we assume that the 
speed measure $m(x)$ of the diffusion $(X_s)_{s \geq 0}$ behaves like a power of $x$ at infinity. 
The  proof of Theorem \ref{penalisation} is divided into several steps, each of them 
corresponding to a lemma or a proposition. The first step is the following:
\begin{lemma} \label{pen1}
For all $a > 0$, $s > 0$, $u \geq 0$, and for all bounded, $\mathcal{F}_s$-measurable, nonnegative functionals
$\Gamma_s$:
$$\lambda^{1-\alpha} \int_u^{\infty} e^{-\lambda t}  \mathbb{P} \left[ \Gamma_s \mathds{1}_{g_t^{[a]} \leq s} \right] 
\, dt  \underset{\lambda \rightarrow 0}{\longrightarrow} K \, \mathcal{Q} \left[\Gamma_s \mathds{1}_{g^{[a]}
 \leq s} \right] < \infty,$$
where $K$ is the 
constant introduced in the last item of Lemma \ref{exp}. 
\end{lemma}
\begin{proof}
One has by the  Markov property that
\begin{align*}
\lambda^{1-\alpha} \int_s^{\infty} e^{-\lambda t} \mathbb{P} \left[ \Gamma_s \mathds{1}_{g_t^{[a]} \leq s} \right] 
\, dt & = \lambda^{1-\alpha} \int_s^{\infty} e^{-\lambda t} \mathbb{P} \left[ \Gamma_s \, \mathds{1}_{X_s \geq a} 
\mathbb{P}_{X_s} \left[ T_a > t-s \right] \right]
\\ & = \lambda^{1-\alpha} e^{-\lambda s} \mathbb{P} \left[ \Gamma_s \, \mathds{1}_{X_s \geq a} 
\mathbb{P}_{X_s} \left[ \int_0^{\infty} e^{-\lambda u} \mathds{1}_{T_a > u} du \right] \right] 
\\ & = \lambda^{1-\alpha} e^{-\lambda s} \mathbb{P} \left[ \Gamma_s \, \mathds{1}_{X_s \geq a} 
\mathbb{P}_{X_s} \left[ \frac{1-e^{-\lambda T_a} }{\lambda} \right] \right] 
\\ & = \lambda^{-\alpha} e^{-\lambda s} \mathbb{P} \left[ \Gamma_s \mathds{1}_{X_s \geq a} 
\frac{ \Phi_{\lambda} (a \lambda^{\alpha} ) - \Phi_{\lambda} (X_s \lambda^{\alpha} ) }
{\Phi_{\lambda} (a \lambda^{\alpha} ) } \right] 
\\ & = \frac{ e^{-\lambda s} }{\Phi_{\lambda} (a \lambda^{\alpha} ) }
\, \mathbb{P} \left[ \Gamma_s (X_s-a)_+ \int_0^1 dv \, \left| \Phi'_{\lambda} \left( \left( a+ v(X_s-a)_+ \right)
\lambda^{\alpha} \right) \right| \right].
\end{align*}
Now, by last item of Lemma \ref{exp}, for $v$, $a$, $X_s$ fixed:
$$\left|\Phi'_{\lambda} \left( \left( a+ v(X_s-a)_+ \right)
\lambda^{\alpha} \right) \right| \underset{\lambda \rightarrow 0}{\longrightarrow} K.$$
Moreover, by the fourth item, for $\lambda \leq \lambda_0$:
$$\left|\Phi'_{\lambda} \left( \left( a+ v(X_s-a)_+ \right)
\lambda^{\alpha} \right) \right| \leq K',$$
where $K' > 0$ does not depend on $\lambda$, $v$, $a$ or $X_s$.
Since $\mathbb{P} [\Gamma_s (X_s-a)_+]$ is dominated by $\mathbb{P} [X_s] < \infty$, one can apply dominated 
convergence, which yields
$$\lambda^{1-\alpha} \int_s^{\infty} e^{-\lambda t} \mathbb{P} \left[ \Gamma_s \mathds{1}_{g_t^{[a]} \leq s} \right]
\underset{\lambda \rightarrow 0}{\longrightarrow} K  \, \mathbb{P} [\Gamma_s (X_s-a)_+],$$ 
and Lemma \ref{pen1} for $u =s$. Now, if $M >0$ majorizes uniformly $\Gamma_s$, we have
$$\lambda^{1-\alpha} \left|\int_s^{u} e^{-\lambda t}
 \mathbb{P} \left[ \Gamma_s \mathds{1}_{g_t^{[a]} \leq s} \right] \right| 
\leq M |u-s| \, \lambda^{1-\alpha},$$
which tends to zero with $\lambda$. 
\end{proof}
\noindent
The next step is the following:
\begin{lemma} \label{pen2}
Let $a > 0$ be fixed. Then, there exist $u_0 >0$, $L > 0$ such that for all $r > 0$, 
$u \geq u_0$ and for all $\mathcal{F}_r$-measurable, bounded, nonnegative functionals $\Gamma_r$:
$$\mathbb{P} [\Gamma_r \mathds{1}_{g_{r+u}^{[a]} \leq r} ] \leq L u^{-\alpha} \mathcal{Q} 
[\Gamma_r \mathds{1}_{g^{[a]}
\leq r} ].$$
\end{lemma}
\begin{proof}
One has:
\begin{align*}
\mathbb{P} [\Gamma_r \mathds{1}_{g_{r+u}^{[a]} \leq r} ] & = 
\mathbb{P} \left[\Gamma_r \mathds{1}_{X_r \geq a} \mathbb{P}_{X_r} [T_a > u] \right] \\ & 
\leq \frac{1}{1-(1/e)} \, \mathbb{P} \left[ \Gamma_r \mathds{1}_{X_r \geq a} \mathbb{P}_{X_r} [1-e^{-T_a/u}] \right] 
\\ & \leq 2 \, \mathbb{P} \left[\Gamma_r \mathds{1}_{X_r \geq a}  \frac{ \Phi_{1/u} (au^{-\alpha})
- \Phi_{1/u} (X_r u^{-\alpha} ) } {\Phi_{1/u} (au^{-\alpha}) } \right] 
\\ & \leq \frac{2}{\Phi_{1/u} (au^{-\alpha}) } \, u^{-\alpha}
\, \mathbb{P} \left[ \Gamma_r (X_r-a)_+ \int_0^1 dv \, \left| \Phi'_{1/u} \left( \left( a+ v(X_r-a)_+ \right)
u^{-\alpha} \right) \right| \right].
\end{align*}
\noindent
From the fourth item of Lemma \ref{exp}, one deduces that $|\Phi'_{1/u}|$ is uniformly bounded 
by a constant $L'$ for $u \geq 1/\lambda_0.$ Under these assumptions:
$$\mathbb{P} [\Gamma_r \mathds{1}_{g_{r+u}^{[a]} \leq r} ] \leq 
\frac{2L'}{\Phi_{1/u} (au^{-\alpha}) } \, u^{-\alpha} \, \mathbb{P} [\Gamma_r (X_r-a)_+ ]
\leq 3L'u^{-\alpha} \mathcal{Q} [\Gamma_r \mathds{1}_{g^{[a]} \leq r} ],$$
for $u$ sufficiently large in order to make sure that $\Phi_{1/u} (au^{-\alpha} ) \geq 2/3$.
\end{proof}
\noindent
Now, let us prove the following:
\begin{lemma} \label{pen3}
Let $a > 0$ and let $(\Gamma_t)_{t \geq 0}$ be a c\`adl\`ag,
 adapted, nonnegative, uniformly bounded and nonincreasing process, such that for all $t \geq 0$, $X_t = 
X_{g^{[a]}}$ on the set $\{t \geq g^{[a]}\}$. We 
define $\Gamma_{\infty}$  as the limit of $\Gamma_t$ for $t$ going to infinity (in particular, 
$\Gamma_{\infty} = \Gamma_{g^{[a]}}$ 
for $g^{[a]} < \infty$), and we suppose that $\Gamma_{\infty}$ is integrable with respect to $\mathcal{Q}$.
Then, for all $\gamma \in (0,1/2)$, there exists $R > 0$ such that for all $s \geq 0$:
$$\underset{\lambda \rightarrow 0}{\lim \, \sup} \, \lambda^{1-\alpha}
\int_0^{\infty} e^{-\lambda t}  \mathbb{P} \left[ \Gamma_t \mathds{1}_{s < g_t^{[a]} \leq (1-\gamma)t} \right] 
\, dt  \leq R \, \mathcal{Q} [\Gamma_{\infty} \mathds{1}_{g^{[a]} > s}].$$
\end{lemma}
\begin{proof}
For all $t > 0$:
$$\mathbb{P} \left[ \Gamma_t \mathds{1}_{s < g_t^{[a]} \leq (1-\gamma)t} \right] 
\leq \mathbb{P} \left[ \Gamma_{(1-\gamma)t} \mathds{1}_{g_{(1-\gamma)t}^{[a]} > s} \mathds{1}_{g_t^{[a]} 
\leq (1-\gamma)t} \right].$$
Remark that the quantities involved here are equal to zero for $t \leq s/(1-\gamma)$. 
By Lemma \ref{pen2}, one deduces that for $t \geq u_0/\gamma$:
$$\mathbb{P} \left[ \Gamma_t \mathds{1}_{s < g_t^{[a]} \leq (1-\gamma)t} \right] 
\leq L(\gamma t)^{-\alpha} \mathcal{Q} \left[\Gamma_{(1-\gamma)t} \mathds{1}_{g_{(1-\gamma)t}^{[a]} > s}
 \mathds{1}_{g^{[a]} \leq (1-\gamma)t} \right].$$
Now, if $g^{[a]} \leq (1-\gamma) t$, $\Gamma_{(1-\gamma)t} = \Gamma_{\infty}$, which implies:
$$\mathbb{P} \left[ \Gamma_t \mathds{1}_{s < g_t^{[a]} \leq (1-\gamma)t} \right] 
\leq L(\gamma t)^{-\alpha} \mathcal{Q} [\Gamma_{\infty} \mathds{1}_{g^{[a]} > s} ].$$
One now deduces:
$$\lambda^{1-\alpha} \int_{u_0/\gamma}^{\infty} e^{-\lambda t}
 \mathbb{P} \left[ \Gamma_t \mathds{1}_{s < g_t^{[a]} \leq (1-\gamma)t} \right] \, dt
\leq L  \, \Gamma(1-\alpha) \, \gamma^{-\alpha}  
 \mathcal{Q} [\Gamma_{\infty} \mathds{1}_{g^{[a]} > s} ].$$
Since
$$\lambda^{1-\alpha} \int_0^{u_0/\gamma} e^{-\lambda t}
 \mathbb{P} \left[ \Gamma_t \mathds{1}_{s < g_t^{[a]} \leq (1-\gamma)t} \right]
\underset{\lambda \rightarrow 0}{\longrightarrow} 0,$$
we are done.
\end{proof}
\begin{lemma} \label{pen4}
Let $a > 0$ and let $(\Gamma_t)_{t \geq 0}$ be a c\`adl\`ag,
 adapted, nonnegative, uniformly bounded and nonincreasing process, such that for all $t \geq 0$, $X_t = 
X_{g^{[a]}}$ on the set $\{t \geq g^{[a]}\}$. We 
define $\Gamma_{\infty}$  as the limit of $\Gamma_t$ for $t$ going to infinity, and we suppose
 that $\Gamma_{\infty}$ is integrable with respect to $\mathcal{Q}$.
Then, for all $\gamma \in (0,1/2)$, for all $s \geq 0$ and for all events $\Lambda_s \in \mathcal{F}_s$:
$$\lambda^{1-\alpha}
\int_0^{\infty} e^{-\lambda t}  \mathbb{P} \left[ \Gamma_t \mathds{1}_{\Lambda_s} \mathds{1}_{ g_t^{[a]}
 \leq (1-\gamma)t} \right] 
\, dt  \underset{\lambda \rightarrow 0}{\longrightarrow}
 K \, \mathcal{Q} [\Gamma_{\infty} \mathds{1}_{\Lambda_s}].$$
\end{lemma}
\begin{proof}
Let $v > s$. For $t > v$:
\begin{align*}
\mathbb{P} [\Gamma_t \mathds{1}_{\Lambda_s} \mathds{1}_{g_t^{[a]} \leq v} ]
& = \mathbb{P} \left[ \frac{\Gamma_t \mathds{1}_{\Lambda_s} \mathds{1}_{g_t^{[a]} \leq v} \mathds{1}_{X_t > a}}
{X_t-a} \, (X_t-a)_+ \right] \\ & 
= \mathcal{Q} \left[ \frac{\Gamma_t \mathds{1}_{\Lambda_s} \mathds{1}_{g_t^{[a]} \leq v} \mathds{1}_{X_t > a}}
{X_t-a}  \, \mathds{1}_{g^{[a]} \leq t} \right]
\\ & = \mathcal{Q} \left[ \frac{\Gamma_t \mathds{1}_{\Lambda_s} \mathds{1}_{X_t > a}}
{X_t-a}  \, \mathds{1}_{g^{[a]} \leq v} \right]
\\ & =  \mathcal{Q} \left[ \frac{\Gamma_v \mathds{1}_{\Lambda_s} \mathds{1}_{X_t > a}}
{X_t-a}  \, \mathds{1}_{g^{[a]} \leq v} \right]
\\ & = \mathcal{Q} \left[ \frac{\Gamma_v \mathds{1}_{\Lambda_s} \mathds{1}_{g_t^{[a]} \leq v} \mathds{1}_{X_t > a}}
{X_t-a}  \, \mathds{1}_{g^{[a]} \leq t} \right]
\\ & = \mathbb{P} \left[ \frac{\Gamma_v \mathds{1}_{\Lambda_s} \mathds{1}_{g_t^{[a]} \leq v} \mathds{1}_{X_t > a}}
{X_t-a} \, (X_t-a)_+ \right] \\ & 
= \mathbb{P} [\Gamma_v \mathds{1}_{\Lambda_s} \mathds{1}_{g_t^{[a]} \leq v} ]
\end{align*}
\noindent
By Lemma \ref{pen1}, one has
$$\lambda^{1-\alpha}
\int_v^{\infty} e^{-\lambda t}  \mathbb{P} \left[ \Gamma_t \mathds{1}_{\Lambda_s} \mathds{1}_{ g_t^{[a]}
 \leq v} \right] 
\, dt  \underset{\lambda \rightarrow 0}{\longrightarrow}
 K \, \mathcal{Q} [\Gamma_{v} \mathds{1}_{\Lambda_s} \mathds{1}_{g^{[a]} \leq v}],$$
which implies
$$\lambda^{1-\alpha}
\int_w^{\infty} e^{-\lambda t}  \mathbb{P} \left[ \Gamma_t \mathds{1}_{\Lambda_s} \mathds{1}_{ g_t^{[a]}
 \leq v} \right] \, dt  \underset{\lambda \rightarrow 0}{\longrightarrow}
 K \, \mathcal{Q} [\Gamma_{\infty} \mathds{1}_{\Lambda_s} \mathds{1}_{g^{[a]} \leq v}]$$
for all $w \geq 0$.
Now, one also has:
$$\int_0^{\infty} e^{-\lambda t}  \mathbb{P} \left[ \Gamma_t \mathds{1}_{\Lambda_s} \mathds{1}_{ g_t^{[a]}
 \leq (1-\gamma)t} \right] 
\, dt  \geq 
\int_{v/(1-\gamma)}^{\infty} e^{-\lambda t}  \mathbb{P} \left[ \Gamma_t \mathds{1}_{\Lambda_s} \mathds{1}_{ g_t^{[a]}
 \leq v} \right] 
\, dt,$$
which implies:
\begin{equation}
\underset{\lambda \rightarrow 0}{\lim \, \inf} \, \lambda^{1-\alpha} \int_0^{\infty} e^{-\lambda t}  
\mathbb{P} \left[ \Gamma_t \mathds{1}_{\Lambda_s} \mathds{1}_{ g_t^{[a]}
 \leq (1-\gamma)t} \right] \, dt  \geq K \, \mathcal{Q} [\Gamma_{\infty} 
\mathds{1}_{\Lambda_s} \mathds{1}_{g^{[a]} \leq v}]. \label{liminf}
\end{equation} 
\noindent
On the other hand:
\begin{align*}
\int_0^{\infty} e^{-\lambda t}  \mathbb{P} \left[ \Gamma_t \mathds{1}_{\Lambda_s} \mathds{1}_{ g_t^{[a]}
 \leq (1-\gamma)t} \right] 
\, dt & \leq \int_0^{\infty} e^{-\lambda t}  \mathbb{P} \left[ \Gamma_t \mathds{1}_{\Lambda_s} \mathds{1}_{ g_t^{[a]}
 \leq v} \right] \, dt \\ & +
\int_0^{\infty} e^{-\lambda t}  \mathbb{P} \left[ \Gamma_t \mathds{1}_{ v < g_t^{[a]}
 \leq (1-\gamma)t} \right] \, dt,
\end{align*}
which implies by Lemma \ref{pen3} that
\begin{equation}
\underset{\lambda \rightarrow 0}{\lim \, \sup} \, \lambda^{1-\alpha} \int_0^{\infty} e^{-\lambda t}  
\mathbb{P} \left[ \Gamma_t \mathds{1}_{\Lambda_s} \mathds{1}_{ g_t^{[a]}
 \leq (1-\gamma)t} \right] \, dt  \leq K \, \mathcal{Q} [\Gamma_{\infty} 
\mathds{1}_{\Lambda_s} \mathds{1}_{g^{[a]} \leq v}] + R \, \mathcal{Q} [\Gamma_{\infty} 
 \mathds{1}_{g^{[a]} > v}]. \label{limsup}
\end{equation}
\noindent
By comparing \eqref{liminf} and \eqref{limsup} and by taking $v \rightarrow \infty$, we are done, since
$\Gamma_{\infty}$ is $\mathcal{Q}$-integrable and $g^{[a]} < \infty$, $\mathcal{Q}$-almost everywhere. 
\end{proof}
\noindent
The following result gives, for all $s \geq 0$ and all events $\Lambda_s \in \mathcal{F}_s$, an equivalent
of the expectation of $\Gamma_t \mathds{1}_{\Lambda_s}$ under $\mathbb{P}$, 
for $t$ going to infinity, if $(\Gamma_t)_{t \geq 0}$ belongs to the class (C). 
\begin{proposition} \label{pen5}
Let $(\Gamma_t)_{t \geq 0}$ be a process in the class (C). 
Then, there exists $D>0$ such that for all $s \geq 0$ and for all events $\Lambda_s \in \mathcal{F}_s$:
$$t^{\alpha} \, \mathbb{P} [\Gamma_t \mathds{1}_{\Lambda_s}] \underset{t \rightarrow \infty}{\longrightarrow}
 D \, \mathcal{Q} [ \Gamma_{\infty} \mathds{1}_{\Lambda_s} ].$$
\end{proposition}
\begin{proof}
Lemma \ref{pen4} implies immediately (by taking any $\gamma$) that
\begin{equation} 
\underset{\lambda \rightarrow 0}{\lim \, \inf} \, \lambda^{1-\alpha}
\int_0^{\infty} e^{-\lambda t}  \mathbb{P} [\Gamma_t \mathds{1}_{\Lambda_s}] \, dt 
\geq K \, \mathcal{Q} [ \Gamma_{\infty} \mathds{1}_{\Lambda_s} ]. \label{liminf2}
\end{equation}
\noindent
Now, for $\lambda > 0$, let us define:
$$I(\lambda) :=  \lambda^{1-\alpha}
\int_0^{\infty} e^{-\lambda t}  \mathbb{P} [\Gamma_t] \, dt$$
and for $\lambda > 0$, $\gamma < (0,1/2)$,
$$I(\lambda, \gamma) := \lambda^{1-\alpha}
\int_0^{\infty} e^{-\lambda t}  \mathbb{P} [\Gamma_t \mathds{1}_{g_t^{[a]} \leq (1-\gamma)t}] \, dt$$
and 
$$J(\lambda, \gamma) := \lambda^{1-\alpha}
\int_0^{\infty} e^{-\lambda t}  \mathbb{P} [\Gamma_t \mathds{1}_{g_t^{[a]} > (1-\gamma)t}] \, dt.$$
Here $a > 0$ is chosen in order to have $\Gamma_t = \Gamma_{g_t^{[a]}}$ on the event $\{t \geq g_t^{[a]}\}$. 
For all $t \geq 0$, by the Markov property
$$\mathbb{P}[\Gamma_t \mathds{1}_{g_t^{[a]} > (1-\gamma)t} ] 
\leq \mathbb{P} [\Gamma_{t/2}] \, \sup_{x \in \mathbb{R}_+} \, \mathbb{P}_x \left[ \exists u \in [(1/2 - \gamma) t,t/2],
X_u \leq a \right].$$
By Lemma \ref{gamma}, for $t \geq 2 t_0 (a, 2 \gamma)$:
$$\mathbb{P}[\Gamma_t \mathds{1}_{g_t^{[a]} > (1-\gamma)t} ]  \leq  \rho(2 \gamma) \, \mathbb{P} [\Gamma_{t/2}].$$
If $M > 0$ majorizes uniformly $\Gamma_t$ for all $t \geq 0$, one deduces:
$$\int_{0}^{\infty}  e^{-\lambda t}  \mathbb{P} [\Gamma_t \mathds{1}_{g_t^{[a]} > (1-\gamma)t}] \, dt
\leq 2 M t_0(a, 2 \gamma) + \rho(2\gamma) \, \int_0^{\infty} e^{-\lambda t} \mathbb{P}[\Gamma_{t/2}] \, dt, $$
which implies:
$$J(\lambda, \gamma) \leq 2 M \lambda^{1-\alpha}  t_0(a, 2 \gamma) + 2 \rho(2\gamma) I(\lambda).$$
Therefore:
$$I(\lambda) \leq I(\lambda, \gamma) + 2 M \lambda^{1-\alpha}  t_0(a, 2 \gamma) + 2 \rho(2\gamma) I(\lambda),$$
which implies
$$I(\lambda) \leq \frac{I(\lambda, \gamma) + 2 M \lambda^{1-\alpha}  t_0(a, 2 \gamma) }{1 - 2 \rho(2\gamma)},$$
and finally
$$J(\lambda, \gamma) \leq 2 M \lambda^{1-\alpha}  t_0(a, 2 \gamma)  + 2 \rho(2\gamma) \, 
\frac{I(\lambda, \gamma) + 2 M \lambda^{1-\alpha}  t_0(a, 2 \gamma) }{1 - 2 \rho(2\gamma)},$$
for $\gamma$ such that $\rho(2\gamma) < 1/2$ (this condition is always satisfied if $\gamma$ is small enough).
Since, by Lemma \ref{pen4}, 
$$I(\lambda, \gamma) \underset{\lambda \rightarrow 0}{\longrightarrow}  K \, \mathcal{Q} [ \Gamma_{\infty} ],$$
one deduces:
$$\underset{\lambda \rightarrow 0}{\lim \, \sup} \, J(\lambda, \gamma)  
\leq \frac{2 K  \rho(2\gamma) }{1 - 2 \rho(2\gamma)} \,\mathcal{Q} [ \Gamma_{\infty}],$$
and then,
\begin{align*}
\underset{\lambda \rightarrow 0}{\lim \, \sup} \, \lambda^{1-\alpha}
\int_0^{\infty} e^{-\lambda t}  \mathbb{P} [\Gamma_t \mathds{1}_{\Lambda_s}] \, dt 
& \leq \underset{\lambda \rightarrow 0}{\lim \, \sup} \, \lambda^{1-\alpha}
\int_0^{\infty} e^{-\lambda t}  \mathbb{P} [\Gamma_t \mathds{1}_{\Lambda_s} \mathds{1}_{g_t^{[a]} \leq (1-\gamma)t}]
\\ & \; +   \underset{\lambda \rightarrow 0}{\lim \, \sup}  \,  J(\lambda, \gamma)  \\ & 
 \leq K \, \mathcal{Q} [ \Gamma_{\infty} \mathds{1}_{\Lambda_s}] + 
\frac{2 K  \rho(2\gamma) }{1 - 2 \rho(2\gamma)} \,\mathcal{Q} [ \Gamma_{\infty}].
\end{align*}
\noindent
By making $\gamma \rightarrow 0$, one deduces:
\begin{equation}
\underset{\lambda \rightarrow 0}{\lim \, \sup} \, \lambda^{1-\alpha}
\int_0^{\infty} e^{-\lambda t}  \mathbb{P} [\Gamma_t \mathds{1}_{\Lambda_s}] \, dt \leq 
K \, \mathcal{Q} [ \Gamma_{\infty} \mathds{1}_{\Lambda_s}] \label{limsup2}
\end{equation}
\noindent
By taking \eqref{liminf2} and \eqref{limsup2} together and by using Tauberian theorem (recall that $\Gamma_t$ is 
decreasing with respect to $t$), we are done. 
\end{proof}
\noindent
Now, Theorem \ref{penalisation} can be deduced from Proposition \ref{pen5} in a very simple way, 
as follows. Let us suppose that for some $t \geq 0$, $\Gamma_t = 0$ almost surely with respect
 to $\mathbb{P}$. 
Since all the $\mathbb{P}$-negligible events in $\mathcal{F}_t$ are also $\mathcal{Q}$-negligible,
one has $\mathcal{Q} [\Gamma_t] = 0$, which contradicts the fact that $\mathcal{Q}[\Gamma_{\infty}] > 0$, since
$\Gamma_t$ decreases with respect to $t$. Then, $\mathbb{P} [\Gamma_t] > 0$ for all $t \geq 0$, and 
obviously $\mathbb{P} [\Gamma_t] < \infty$, since $\Gamma_t$ is uniformly bounded. Hence, $\mathbb{Q}_t$ is 
well-defined. Moreover, for all $s \geq 0$ and $\Lambda_s \in \mathcal{F}_s$:
$$\mathbb{Q}_t [\Lambda_s] = \frac{ \mathbb{P} [\Gamma_t \mathds{1}_{\Lambda_s}]} {\mathbb{P} [\Gamma_t]}.$$
Now by Proposition \ref{pen5},
$$\mathbb{P} [\Gamma_t \mathds{1}_{\Lambda_s}] \underset{t \rightarrow \infty}{\sim} \,
D\, t^{-\alpha} \mathcal{Q} [\Gamma_{\infty} \mathds{1}_{\Lambda_s} ]$$
and 
$$\mathbb{P} [\Gamma_t ] \underset{t \rightarrow \infty}{\sim} D t^{-\alpha} \mathcal{Q} [\Gamma_{\infty} ].$$
Therefore:
$$\mathbb{Q}_t [\Lambda_s] \underset{t \rightarrow \infty}{\longrightarrow} 
\frac{\mathcal{Q} [\Gamma_{\infty} \mathds{1}_{\Lambda_s} ]}{ \mathcal{Q} [\Gamma_{\infty}]} = \mathbb{Q}_{\infty} 
[\Lambda_s],$$
which completes the proof of Theorem \ref{penalisation}. 

\providecommand{\bysame}{\leavevmode\hbox to3em{\hrulefill}\thinspace}
\providecommand{\MR}{\relax\ifhmode\unskip\space\fi MR }
\providecommand{\MRhref}[2]{%
  \href{http://www.ams.org/mathscinet-getitem?mr=#1}{#2}
}
\providecommand{\href}[2]{#2}

\end{document}